\documentclass[12pt]{article}
\usepackage{amsmath,amssymb,amsthm}
\usepackage[margin=1in]{geometry}
\usepackage{hyperref}
\usepackage{tikz}
\usepackage{graphicx}
\usepackage{url}
\usepackage[mathlines]{lineno}
\usepackage{multirow}
\usepackage{dsfont} % needed for double-struck 1
\usepackage{color}
\definecolor{red}{rgb}{1,0,0}

\newtheorem{thm}{Theorem}[section]
\newtheorem{cor}[thm]{Corollary}
\newtheorem{lem}[thm]{Lemma}
\newtheorem{prop}[thm]{Proposition}

\newtheorem{obs}[thm]{Observation}

\newtheorem{quest}[thm]{Question}

\newtheorem{rem}[thm]{Remark}

\theoremstyle{definition}
\newtheorem{defn}[thm]{Definition}

\theoremstyle{example}
\newtheorem{ex}[thm]{Example}

\def\mtx#1{\begin{bmatrix} #1 \end{bmatrix}}

\DeclareMathOperator{\epr}{epr}

\DeclareMathOperator{\rank}{rank}

\newcommand{\R}{\mathbb{R}}
\newcommand{\F}{\mathbb{F}}

\newcommand{\Fnn}{\F^{n\times n}}

\newcommand{\bit}{\begin{itemize}}
\newcommand{\eit}{\end{itemize}}
\newcommand{\ben}{\begin{enumerate}}
\newcommand{\een}{\end{enumerate}}
\newcommand{\beq}{\begin{equation}}
\newcommand{\eeq}{\end{equation}}
\newcommand{\bea}{\begin{eqnarray*}}
\newcommand{\eea}{\end{eqnarray*}}
\newcommand{\bpf}{\begin{proof}}
\newcommand{\epf}{\end{proof}\ms}
\newcommand{\bmt}{\begin{bmatrix}}
\newcommand{\emt}{\end{bmatrix}}
\newcommand{\ms}{\medskip}

\newcommand{\ba}{\begin{array}}
\newcommand{\ea}{\end{array}}

\usepackage{enumitem}
\DeclareMathOperator{\qpr}{qpr}

\DeclareMathOperator{\apr}{apr}

\DeclareMathOperator{\aprank}{ap-rank}

\begin{document}
%The almost-principal rank characteristic sequence
\title{On the almost-principal minors of a symmetric matrix}
\author{
Shaun M. Fallat\thanks{Department of Mathematics and Statistics,
University of Regina,  Regina, Saskatchewan, S4S 0A2, Canada  (shaun.fallat@uregina.ca).}
\and
Xavier Mart\'inez-Rivera\thanks{Department of Mathematics and Statistics,
Auburn University, Auburn, AL 36849, USA (martinez.rivera.xavier@gmail.com).
}
}

%\linenumbers

\maketitle

\begin{abstract}
%The almost-principal rank of a symmetric matrix $B$,
%denoted by $\aprank(B)$, is defined as the size of a
%largest nonsingular almost-principal submatrix of $B$.
The almost-principal rank characteristic sequence (apr-sequence) of
an  $n\times n$ symmetric matrix is introduced, 
which is defined to be the string 
$a_1 a_2 \cdots a_{n-1}$, where $a_k$ is 
either {\tt A}, {\tt S}, or {\tt N},
according as all, some but not all, or none of
its almost-principal minors of order $k$ are nonzero.
In contrast to the other principal rank characteristic sequences in
the literature, the apr-sequence of a matrix does not
depend on principal minors.
The almost-principal rank of a symmetric matrix $B$,
denoted by $\aprank(B)$, is defined as the size of a
largest nonsingular almost-principal submatrix of $B$.
A complete characterization of the sequences not
containing an $\tt A$ that can be realized as
the apr-sequence of a symmetric matrix over a field $\F$ is provided.
A necessary condition for a sequence to
be the apr-sequence of a symmetric matrix over a field $\F$ is presented.
%and it is conjectured that this condition is sufficient if $\F$ is of 
%characteristic $0$.
It is shown that if $B \in \Fnn$ is symmetric and non-diagonal, then
$\rank(B)-1 \leq \aprank(B) \leq \rank(B)$, with both bounds being sharp.
Moreover, it is shown that if $B$ is symmetric, non-diagonal and singular, and does not contain a zero row, then $\rank(B) = \aprank(B)$.
\end{abstract}

\noindent{\bf Keywords.}
Almost-principal minor; 
almost-principal rank characteristic sequence;
enhanced principal rank characteristic sequence;
symmetric matrix;
rank;
ap-rank.

\medskip

\noindent{\bf AMS subject classifications.}
15B57, 15A15,  15A03.

%%%%%%%%%%%%%%%%%%%%%%%%%%%%%%%
%%%%%%%%%%%%%%%%%%%%%%%%%%%%%%%
%%%%%%%%%%%%%%%%%%%%%%%%%%%%%%%
%%%%%%%%%%%%%%%%%%%%%%%%%%%%%%%
%%%%%%%%%%%%%%%%%%%%%%%%%%%%%%%
%%%%%%%%%%%%%%%%%%%%%%%%%%%%%%%
%%%%%%%%%%%%%%%%%%%%%%%%%%%%%%%
%%%%%%%%%%%%%%%%%%%%%%%%%%%%%%%
%\pagebreak
\section{Introduction}\label{s: intro}
$\null$
\indent
Motivated by work of Brualdi et al.\ \cite{P} on the
principal rank characteristic sequence (pr-sequence),
Butler et al.\ \cite{EPR} introduced the
enhanced principal rank characteristic sequence (epr-sequence),
which they defined as follows:
For a given symmetric matrix $B \in \Fnn$, where $\F$ is a field,
the \textit{enhanced principal rank characteristic sequence}
(epr-sequence) of $B$ is
$\epr(B) = \ell_1\ell_2 \cdots \ell_n$, where
\begin{equation*}
   \ell_k =
    \begin{cases}
             \tt{A} &\text{if all of the principal minors of order $k$ are nonzero;}\\
             \tt{S} &\text{if some but not all of the principal minors of order $k$ are nonzero;}\\
             \tt{N} &\text{if none of the principal minors of order $k$ are nonzero (i.e., all are zero),}
         \end{cases}
\end{equation*}
where a minor of {\em order} $k$ is the determinant of a
$k \times k$ submatrix of $B$.
(The definition of ``principal minor'' appears below.)
After subsequent work on epr-sequences
(see \cite{{EPR-Hermitian}, {skew}, {XMR-Classif}, {XMR-Char 2}}),
another sequence, one that refines the epr-sequence, called the
signed enhanced principal rank characteristic sequence (sepr-sequence),
was introduced by Mart\'inez-Rivera in \cite{XMR-sepr}.
Recently, Fallat and Mart\'inez-Rivera \cite{qpr} extended the definition of the epr-sequence by also taking into consideration the
almost-principal minors of the matrix,
leading them to a new sequence,
which we will define after introducing some terminology:
For  $B \in \Fnn$ and $\alpha, \beta \subseteq \{1,2, \dots, n\}$,
$B[\alpha, \beta]$ will denote the submatrix lying in
rows indexed by $\alpha$ and columns indexed by $\beta$;
$B[\alpha, \beta]$ is a {\em principal} submatrix of $B$ if
$\alpha = \beta$;
the minor $\det B[\alpha,\beta]$ is a {\em principal} minor of $B$ if
$B[\alpha, \beta]$ is a principal submatrix of $B$;
$B[\alpha, \beta]$ is an {\em almost-principal} submatrix of $B$ if
$|\alpha| = |\beta|$ and $|\alpha \cap \beta| = |\alpha|-1$;
the minor $\det B[\alpha,\beta]$ is an {\em almost-principal} minor of $B$ if
$B[\alpha, \beta]$ is an  almost-principal submatrix of $B$;
the minor $\det B[\alpha,\beta]$ is a {\em quasi-principal} minor of $B$ if
$B[\alpha, \beta]$ is a principal or an almost-principal submatrix of $B$;
we will say that an $n \times n$ matrix has {\em order} $n$;
a sequence $t_1t_2 \cdots t_{k}$ from
$\{\tt A,N,S\}$ is said to have {\em length} $k$.
As introduced in \cite{qpr},
for a given symmetric matrix $B \in \Fnn$, where $\F$ is a field,
the \textit{quasi principal rank characteristic sequence}
(qpr-sequence) of $B$ is
$\qpr(B)=q_1q_2\cdots q_n$, where
\[q_k=\left\{\begin{array}{ll}
{\tt A} &  \mbox{ if  all of the quasi-principal minors of order $k$ are nonzero;}\\
{\tt S} &   \mbox{ if some but not all of the quasi-principal minors of order $k$ are nonzero;}\\
{\tt N} &  \mbox{ if none of the quasi-principal minors of order $k$ are nonzero (i.e., all are zero).}\end{array}\right.\]

A necessary condition for a sequence to be the qpr-sequence of a
symmetric matrix over a field $\F$ was found in \cite{qpr}:

\begin{thm}\label{qpr-necessary condition}
{\rm \cite[Corollary 2.7]{qpr}}
Let $\F$ be a field and $q_1q_2 \cdots q_{n}$ be
a sequence from $\{\tt A,N,S\}$.
If $q_1q_2 \cdots q_{n}$ is the qpr-sequence of a
symmetric matrix $B \in \Fnn$, then the following statements hold:
\ben
\item[(i)]
$q_n \neq \tt S$.
\item[(ii)]
Neither $\tt NA$ nor $\tt NS$ is a subsequence of $q_1q_2 \cdots q_{n}$.
\een
\end{thm}

The necessary condition in Theorem \ref{qpr-necessary condition} was
shown to be sufficient if $\F$ is of characteristic $0$:

\begin{thm}\label{qpr-char 0}
{\rm \cite[Theorem 3.7]{qpr}}
Let $\F$ be a field of characteristic $0$.
A sequence $q_1q_2 \cdots q_n$ from $\{\tt A,N,S\}$
is the qpr-sequence of a symmetric matrix $B \in \Fnn$
if and only if
the following statements hold:
\ben
\item[(i)]
$q_n \neq \tt S$.
\item[(ii)]
Neither $\tt NA$ nor $\tt NS$ is a subsequence of $q_1q_2 \cdots q_n$.
\een
\end{thm}

Theorem \ref{qpr-char 0} establishes a contrast between
the epr-sequences and qpr-sequences of symmetric matrices,
since a complete characterization such as the one in
Theorem \ref{qpr-char 0} for epr-sequences when
the field $\F$ is not the field of order $2$ is not yet known
(see \cite{XMR-Char 2}).
The absence of such a characterization for epr-sequences is due to
the difficulty in understanding epr-sequences containing
$\tt NA$ or $\tt NS$ as subsequences.
However, in the case of qpr-sequences, this difficulty was overcome,
since Theorem \ref{qpr-necessary condition} states that neither
$\tt NA$ nor $\tt NS$ can occur as a subsequence of
the qpr-sequence of a symmetric matrix \cite{qpr},
regardless of the field;
this raises a question:

\begin{quest}\label{qpr question}
\rm
Should we attribute the fact that neither
$\tt NA$ nor $\tt NS$ can occur as a subsequence of
the qpr-sequence of a symmetric matrix entirely to
the dependence of qpr-sequences on almost-principal minors?
\end{quest}

Question \ref{qpr question}, together with
the applications that almost-principal minors find
in numerous areas, which include
algebraic geometry, statistics, theoretical physics and matrix theory
\cite{qpr} (see, for example, \cite{{Stu17}, {Vanishing Minor Conditions}, {InvM3}, {KP14}, {Stu09}, {Stu16}, {InvM1}}),
is motivation for introducing
the almost-principal rank and
the almost-principal rank characteristic sequence
of a symmetric matrix, which is the focus of this paper:

\begin{defn}{\rm
Let $B \in \Fnn$ be symmetric.
The \textit{almost-principal rank} of $B$, denoted by
$\aprank(B)$, is
\[\aprank(B) :=
\max \{ |\alpha| :
\det (B[\alpha, \beta]) \neq 0, \
|\alpha| = |\beta|
\mbox{ \  and  \ }
|\alpha \cap \beta| = |\alpha|-1\}
\]
(where the maximum over the empty set is defined to be 0).
}\end{defn}

We note that, by definition,
the ap-rank of a $1 \times 1$ matrix is $0$.

\begin{defn}{\rm
For $n \geq 2$,
the {\em almost-principal rank characteristic sequence} of a
symmetric matrix $B\in\Fnn$ is the sequence (apr-sequence)
$\apr(B)=a_1a_2\cdots a_{n-1}$, where
\[a_k=\left\{\begin{array}{ll}
{\tt A} &  \mbox{ if all of the almost-principal minors of order $k$ are nonzero;}\\
{\tt S} &   \mbox{ if some but not all of the almost-principal minors of order $k$ are nonzero;}\\
{\tt N} &  \mbox{ if none of the almost-principal minors of order $k$ are nonzero (i.e., all are zero).}\end{array}\right.\]
}\end{defn}
Some observations highlighting the contrast between apr-sequences and
pr-, epr-, sepr- and qpr-sequences are now in order:
Unlike the other sequences, by definition,
apr-sequences do not depend on principal minors;
moreover, whether or not a matrix is nonsingular is not
revealed by its apr-sequence;
the apr-sequence of a symmetric matrix $B \in \Fnn$ has length $n-1$
---while the epr-, sepr- and qpr-sequence each has length $n$;
furthermore, unlike epr- and qpr-sequences,
apr-sequences may end with $\tt S$.
Another observation is that
the apr-sequence of a $1 \times 1$ matrix is simply the empty list (or word),
and, therefore, wherever the apr-sequence of an
$n \times n$ matrix is involved, we assume that $n \geq 2$.

In the remainder of the present section,
some of the terminology we adopted is introduced,
known results that are used frequently are listed,
and facts about apr-sequences that
will serve as tools in subsequent sections are established.
%%%%%%%%%%%%%%%%%
In Section \ref{s: apr-sequence}, in particular,
we establish a result analogous to Theorem \ref{NN Thm for epr} below
(the $\tt NN$ Theorem for epr-sequences from \cite{EPR}),
as well as a necessary condition for a sequence not containing an $\tt A$
to be the apr-sequence of a symmetric matrix (over an arbitrary field).
%%%%%%%%%%%%%%%%%
Section \ref{s: No As} is devoted mostly to 
apr-sequences not containing an $\tt A$, 
which are completely characterized (for an arbitrary field) in 
Theorem \ref{No As}, and concludes by 
providing a necessary condition for a  sequence 
(from $\{\tt A,N,S\}$) to be the apr-sequence of a symmetric matrix 
(over an arbitrary field).
%which is then conjectured to be sufficient if 
%the field is of characteristic $0$.
%%%%%%%%%%%%%%%%%
Section \ref{s: ap-rank} is focused on the ap-rank of a symmetric matrix (over an arbitrary field), where it is shown, in particular, that
for a symmetric non-diagonal singular matrix $B$ not containing a zero row,
$\rank(B) = \aprank(B)$.
%%%%%%%%%%%%%%%%%
Section \ref{s: final}  has concluding remarks, including 
an answer to Question \ref{qpr question}.

In what follows, unless otherwise stated,
$\F$ is used to denote an arbitrary field.
%%%%%%%%%%%%%%%%%
Given a vector $x$ of length $n$,
$x[\alpha]$ denotes the subvector of $x$ with
entries indexed by $\alpha \subseteq \{1,2, \dots, n\}$.
%%%%%%%%%%%%%%%%%
If the sequence $a_1a_2 \cdots a_{n-1}$ from $\{\tt A,N,S\}$ is
the apr-sequence of a symmetric matrix over $\F$, then
we will say that the sequence is {\em attainable} over $\F$
(or simply that the sequence is {\em attainable},
if what is meant is clear from the context).
%%%%%%%%%%%%%%%%%
Given a sequence $t_{i_1}t_{i_2} \cdots t_{i_k}$,
$\overline{t_{i_1}t_{i_2} \cdots t_{i_k}}$ indicates that
the sequence may be repeated as many times as desired
(or it may be omitted entirely).
%%%%%%%%%%%%%%%%%
The matrices $B$ and $C$ are said to be
{\em permutationally similar} if there exists a
permutation matrix $P$ such that $C=P^TBP$.
%%%%%%%%%%%%%%%%%
If replacing each of the nonzero entries of a
matrix $P$ with a $1$ results in a permutation matrix,
then we will say that $P$ is a \textit{generalized permutation} matrix.
%%%%%%%%%%%%%%%%%%
%The column and row space of a matrix $B$ are denoted by
%$\CS(B)$ and $\RS(B)$, respectively.
%%%%%%%%%%%%%%%%%
The zero matrix, identity matrix and all-$1$s matrix of order $n$
is denoted with $O_n$, $I_n$ and $J_n$, respectively;
moreover, $O_0$, $I_0$ and $J_0$ are understood to be vacuous.
%%%%%%%%%%%%%%%%%
The block diagonal matrix with the matrices $B$ and $C$ on
the diagonal (in that order) is denoted by $B \oplus C$.

\subsection{Known results}
$\null$
\indent
In this section, known results that are used frequently are listed,
of which some have been assigned abbreviated nomenclature.
We start with a well-known fact (see \cite{BIRS13}, for example),
which states that the rank of a symmetric matrix $B$ is equal to the order
of a largest nonsingular principal submatrix of $B$;
because of this, we will call the rank of a symmetric matrix \textit{principal}.

\begin{thm}
\label{thm: rank of a symm mtx}
{\rm \cite[Theorem 1.1]{BIRS13}}
Let $B \in \Fnn$ be symmetric.
Then $\rank(B) = \max \{ |\gamma| : \det (B[\gamma]) \neq 0 \}$.
%(where the maximum over the empty set is defined to be 0).
\end{thm}

For a given matrix $B$ having a nonsingular principal
submatrix $B[\gamma]$, we denote by $B/B[\gamma]$ the
Schur complement of $B[\gamma]$ in $B$ (see \cite{Schur}).
The following result is also a well-known fact
(see \cite{Brualdi & Schneider}).

\begin{thm}
\label{schur}
{\rm (Schur Complement Theorem.)}
Let $B \in \Fnn$ be symmetric with
$\rank(B)=r$.
Let $B[\gamma]$ be a nonsingular principal
submatrix of $B$ with $|\gamma| = k \leq r$,
and let $C = B/B[\gamma]$.
Then the following statements hold: %\vspace{-3mm}
\begin{enumerate}
\item [$(i)$]\label{p1SC}
$C$ is an $(n-k)\times (n-k)$ symmetric matrix. %\vspace{-3mm}
\item [$(ii)$]\label{p2SC}
Assuming the indexing of $C$ is inherited from $B$,
any minor of $C$ is given by %\vspace{-3mm}
\[ \det C[\alpha, \beta] = \det B[\alpha \cup \gamma, \beta \cup \gamma]/ \det B[\gamma].\]
\item [$(iii)$]\label{p3SC} $\rank(C) = r-k$.%\vspace{-3mm}
%\item [$(iv)$]\label{LHp4SC} Any nonsingular principal submatrix of $B$ of
%order at most $r$ is contained in a nonsingular principal submatrix of order $r$.
\end{enumerate}
\end{thm}

Some necessary results about epr-sequences are listed now.
The following theorem, which appears in \cite{EPR},
follows readily from Jacobi's determinantal identity.

\begin{thm}\label{Inverse Thm}
{\rm \cite[Theorem 2.4]{EPR}}
{\rm (Inverse Theorem for epr-Sequences.)}
Let $B \in \Fnn$ be symmetric and nonsingular.
If
$\epr(B) = \ell_1 \ell_2 \cdots \ell_{n-1}\tt A$, then
$\epr(B^{-1}) = \ell_{n-1}\ell_{n-2} \cdots \ell_1 \tt A$.
\end{thm}

\begin{thm}
\label{NN Thm for epr}
{\rm \cite[Theorem 2.3]{EPR}}
{\rm ($\tt NN$ Theorem for epr-Sequences.)}
Let $B \in \Fnn$ be symmetric.
Suppose that $\epr(B) = \ell_1 \ell_2 \cdots \ell_{n}$ and
$\ell_{k} = \ell_{k+1} = \tt N$ for some $k$.
Then $\ell_j = \tt N$ for all $j \geq k$.
\end{thm}

We now state some facts about qpr-sequences.

\begin{obs}
{\rm \cite[Observation 2.1]{qpr}}
\label{qpr rank}
Let $B \in \Fnn$ be symmetric.
Then $\rank(B)$ is equal to the index of 
the last {\tt A} or {\tt S} in $\qpr(B)$.
\end{obs}

Since the rank of a symmetric matrix is principal,
it is not hard to show that a statement analogous to
Theorem \ref{NN Thm for epr} must hold for qpr-sequences;
however, something stronger does hold:
The next result from \cite{qpr} shows that the presence of a single $\tt N$ in
the qpr-sequence of a symmetric matrix $B \in \Fnn$ implies that the sequence has $\tt N$s from that point forward.
This result is of particular relevance later,
when we show that an analogous statement does not
hold for apr-sequences.

\begin{thm}\label{qpr: N theorem}
{\rm \cite[Theorem 2.6]{qpr}}
Let $B \in \Fnn$ be symmetric.
Suppose that $\qpr(B) = q_1 q_2 \cdots q_n$ and
$q_k = \tt N$ for some $k$.
Then $q_j = \tt N$ for all $j \geq k$.
\end{thm}

We conclude this section with a lemma that is used repeatedly,
which is immediate from the fact that appending a row and column to a matrix of rank $r$ results in a matrix whose rank is at most $r+2$.

\begin{lem}
\label{rank when deleting}
Let $B \in \Fnn$ be nonsingular.
Then the rank of any $(n-1)\times (n-1)$ submatrix of $B$ is at least $n-2$.
\end{lem}

\subsection{Tools for apr-sequences}
$\null$
\indent
Some results that will serve as tools to
establish results in subsequent sections are provided in this section.
The first is an immediate consequence of the
Schur Complement Theorem, and it is, therefore, stated as a corollary.

\begin{cor}\label{schurAN}
{\rm (Schur Complement Corollary.)}
Let $B \in \Fnn$ be symmetric,
$\apr(B)=a_1 a_2 \cdots a_{n-1}$ and
$B[\gamma]$ a nonsingular
principal submatrix of $B$,
with $|\gamma| = k$.
Let $C = B/B[\gamma]$
and $\apr(C)=a'_{1} a'_2 \cdots a'_{n-k-1}$.
Then, for $j=1, \dots, n-k-1$,
$a'_j=a_{j+k}$  if
$a_{j+k} \in \{{\tt A,N}\}$.
\end{cor}

A result analogous to
the Inverse Theorem for epr-Sequences can be established for
apr-sequences, by applying Jacobi's determinantal identity:

\begin{thm}\label{Inverse Thm}
{\rm (Inverse Theorem.)}
Let $B \in \Fnn$ be symmetric and nonsingular.
If
$\apr(B) = a_1a_2 \cdots a_{n-1}$, then
$\apr(B^{-1}) = a_{n-1}a_{n-2} \cdots a_1$.
\end{thm}

Appending a zero row and a zero column to a matrix is a useful operation,
since we can easily determine the apr-sequence of the resulting matrix if
we have the apr-sequence of the original matrix,
which leads us to the next observation,
that is useful when dealing with
sequences that do not contain an $\tt A$.

\begin{obs}\label{append}
Let $B \in \Fnn$ be symmetric,
$\apr(B) = a_1 a_2  \cdots a_{n-1}$ and
$B' = B \oplus O_1$.
Then
$\apr(B') = a'_1 a'_2  \cdots a'_{n-1} \tt N$, with
$a'_j = a_j$ if $a_j = \tt N$, and with
$a'_j = \tt S$ if $a_j \neq \tt N$,
for all $j \in \{1,2, \dots, n-1\}$.
\end{obs}

Another useful tool for working with apr-sequences is the following fact, which is analogous to \cite[Theorem 2.6]{EPR} (Inheritance Theorem for epr-sequences).

\begin{thm}
{\rm (Inheritance Theorem.)}
Let $B\in\Fnn$ be symmetric,
$m \leq n$, and
$1\le k \le m-1$.  
\ben
\item
If $[\apr(B)]_k={\tt N}$, then  $[\apr(C)]_k={\tt N}$ for all $m\times m$ principal submatrices $C$.
\item If  $[\apr(B)]_k={\tt A}$, then  $[\apr(C)]_k={\tt A}$ for all $m\times m$ principal submatrices $C$.
\item If $m \geq 6$ and $k \leq m-5$ and $[\apr(B)]_k = {\tt S}$, then there exists an $m \times m$ principal submatrix $C_S$ such that $[\apr(C_S)]_k ={\tt S}$.
\een

\end{thm}
\begin{proof}
Statements (1) and (2) follow from the fact that an almost-principal submatrix of a principal submatrix of $B$ is also an almost-principal submatrix of $B$.

We now establish the final statement.
Suppose that $m \geq 6$ and $k \leq m-5$ and $[\apr(B)]_k = {\tt S}$.
Let $p_1,p_2,\dots,p_{k-1}, q_1,q_2,\dots,q_{k-1}, i, j ,r ,s
\in \{1,2, \dots, n\}$ be
indices such that the following are almost-principal submatrices of order $k$:
\[
B[\{p_1,p_2,\dots,p_{k-1}, i\},\{p_1,p_2,\dots,p_{k-1}, j\}]
\mbox{ \ and \ }
B[\{q_1,q_2,\dots,q_{k-1}, r\},\{q_1,q_2,\dots,q_{k-1}, s\}];
\]
moreover, assume that
the former submatrix is nonsingular and
the latter is singular.
Without loss of generality, we may assume that
any common indices between the lists
$p_1,p_2,\dots,p_{k-1}$ and $q_1,q_2,\dots,q_{k-1}$
occur in the same position in each list;
moreover, we may assume that these common indices (if any) appear consecutively at the beginning of each list.
If
\[
|\{q_1,q_2,\dots,q_{k-1}\} \cap \{i,j\}| = 2,
\]
then, without loss of generality, assume that
$\{q_{k-2}, q_{k-1}\} = \{i,j\}$.
If
\[
|\{q_1,q_2,\dots,q_{k-1}\} \cap \{i,j\}| = 1,
\]
then, without loss of generality,
assume that $q_{k-1} \in \{i,j\}$.
Consider the following list of almost-principal submatrices of order $k$:
\[
\begin{array}{c}
B[\{p_1,p_2,p_3,\ldots,p_{k-3},p_{k-2},p_{k-1},i\},
\{p_1,p_2,p_3,\ldots,p_{k-3},p_{k-2},p_{k-1}, j\}],\\
%%%%%%%%%%%%%%%
B[\{q_1,p_2,p_3,\ldots,p_{k-3},p_{k-2},p_{k-1},i\},
\{q_1,p_2,p_3,\ldots,p_{k-3},p_{k-2},p_{k-1}, j\}],\\
%%%%%%%%%%%%%%%
B[\{q_1,q_2,p_3,\ldots,p_{k-3},p_{k-2},p_{k-1},i\},
\{q_1,q_2,p_3,\ldots,p_{k-3},p_{k-2},p_{k-1}, j\}],\\
%%%%%%%%%%%%%%%
B[\{q_1,q_2,q_3,\ldots,p_{k-3},p_{k-2},p_{k-1},i\},
\{q_1,q_2,q_3,\ldots,p_{k-3},p_{k-2},p_{k-1}, j\}], \\
%%%%%%%%%%%%%%%
\cdots\\
%%%%%%%%%%%%%%%
B[q_1,q_2,q_3,\ldots,q_{k-3},p_{k-2},p_{k-1},i\},
\{q_1,q_2,q_3,\ldots,q_{k-3},p_{k-2},p_{k-1}, j\}], \\
%%%%%%%%%%%%%%%
B[q_1,q_2,q_3,\ldots,q_{k-3},q_{k-2},q_{k-1},r\},
\{q_1,q_2,q_3,\ldots,q_{k-3},q_{k-2},q_{k-1}, s\}]. \\
\end{array}
\]
Since the first submatrix in the above list is
nonsingular and the last is singular,
this list of submatrices must contain a
nonsingular submatrix and a singular submatrix appearing consecutively, say, $B[\alpha, \beta]$ and $B[\gamma, \mu]$.
Note that $|\alpha \cup \beta \cup \gamma \cup \mu| \leq k+5 \leq m$.
Then by letting $C_S$ be an
$m \times m$ principal submatrix of $B$ containing
$B[\alpha \cup \beta \cup \gamma \cup \mu]$,
the desired conclusion follows.
\end{proof}

%%%%%%%%%%%%%%%%%%%%%%%%%%%%%%%
%%%%%%%%%%%%%%%%%%%%%%%%%%%%%%%
%%%%%%%%%%%%%%%%%%%%%%%%%%%%%%%
%%%%%%%%%%%%%%%%%%%%%%%%%%%%%%%
%%%%%%%%%%%%%%%%%%%%%%%%%%%%%%%
%%%%%%%%%%%%%%%%%%%%%%%%%%%%%%%
%%%%%%%%%%%%%%%%%%%%%%%%%%%%%%%
%%%%%%%%%%%%%%%%%%%%%%%%%%%%%%%
\section{The almost-principal rank characteristic sequence}\label{s: apr-sequence}
$\null$
\indent
We begin with a simple but useful observation.

\begin{obs}
\label{N...}
Let $B \in \Fnn$ be symmetric.
Suppose that $\apr(B) = a_1 a_2 \cdots a_{n-1}$ and
$a_1 = \tt N$.
Then $B$ is a diagonal matrix and $a_j = \tt N$ for all $j \geq 1$.
\end{obs}

\begin{prop}
\label{A...N and S...N}
Let $B \in \Fnn$ be symmetric.
Suppose that
$\apr(B) = {\tt A} a_2 a_3 \cdots a_{n-2}{\tt N}$ or
$\apr(B) = {\tt S} a_2 a_3 \cdots a_{n-2}{\tt N}$.
Then $B$ is singular.
\end{prop}

\bpf
Since $\apr(B)$ does not begin with $\tt N$, $B$ is not a diagonal matrix.
If $B$ was nonsingular, then, since each of its almost-principal minors of order $n-1$ is zero, $B^{-1}$ is a diagonal matrix, which would imply that $B$ itself is a diagonal matrix, leading to a contradiction.
\epf

The $\tt NN$ Theorem for epr-sequences states that if
the epr-sequence of a symmetric matrix $B \in \Fnn$ contains
two consecutive $\tt N$s,
then it must contain $\tt N$s from that point forward;
the same statement holds for apr-sequences:

\begin{thm}
\label{thm:N implies all N}
{\rm ({\tt NN} Theorem.)}
Let $B \in \Fnn$ be symmetric.
Suppose that $\apr(B) = a_1 a_2 \cdots a_{n-1}$ and
$a_{k} = a_{k+1} = \tt N$ for some $k$.
Then $a_j = \tt N$ for all $j \geq k$.
\end{thm}

\bpf
If $k=n-2$, then there is nothing to prove;
thus, assume that $k \leq n-3$.
It suffices to show that $a_{k+2} = \tt N$.
Suppose to the contrary that $a_{k+2} \neq \tt N$.
Let $B[\alpha \cup\{i\}, \alpha \cup\{j\}]$ be
a nonsingular almost-principal submatrix of $B$ with
$|\alpha| = k+1$ (note that $i \neq j$).
We now show that $B$ has a
nonsingular $k \times k$ principal submatrix contained in
the $(k+1)\times (k+1)$ submatrix $B[\alpha]$.
There are two cases:

\noindent
{\bf Case 1}: $B[\alpha]$ is nonsingular.

\noindent
Since $a_{k} = \tt N$, the Inheritance Theorem implies that
every almost-principal minor of $B[\alpha]$ of order $k$ is zero.
Then, as $B[\alpha]$ is a $(k+1) \times (k+1)$ nonsingular matrix,
its inverse is a diagonal matrix, which implies that
$B[\alpha]$ is a (nonsingular) diagonal matrix.
It now follows immediately that $B[\alpha]$ contains 
a nonsingular $k \times k$ principal submatrix, as desired.

\noindent
{\bf Case 2}: $B[\alpha]$ is singular.

\noindent
Since $B[\alpha \cup\{i\}, \alpha \cup\{j\}]$ is a
nonsingular $(k+2)\times (k+2)$ matrix,
Lemma \ref{rank when deleting} implies that
$\rank(B[\alpha]) \geq (k+2)-2 = k$.
Then, as $B[\alpha]$ is $(k+1) \times (k+1)$ singular matrix,
$\rank(B[\alpha]) =  k$.
Since the rank of a symmetric matrix is principal,
$B[\alpha]$ contains a nonsingular $k \times k$ principal submatrix,
as desired.

Let $B[\gamma]$ be a nonsingular $k \times k$ principal submatrix (of $B[\alpha]$) with
$\gamma \subseteq \alpha$ and $|\gamma| = k$.
Then, as $|\gamma|=|\alpha|-1$,
$\alpha = \gamma \cup \{p\}$ for some $p$ (note that $p \neq i,j$).
Let $C=B/B[\gamma]$ and
assume that $C$ inherits the indexing from $B$.
Observe that $C$ is an $(n-k) \times (n-k)$ matrix
(by the Schur Complement Theorem),
and that $n-k \geq 3$.
Suppose that
$\apr(C)=a'_{1} a'_2 \cdots a'_{n-k-1}$
Then, as $a_{k+1} = \tt N$,
the Schur Complement Corollary implies that $a'_{1}  = \tt N$.
It follows from Observation \ref{N...} that
$\apr(C)=\tt NN\overline{N}$.
Hence, $\det C[\{p,i\}, \{p,j\}] = 0$.
However,
by the Schur Complement Theorem,
\begin{align*}
\det C[\{p,i\}, \{p,j\}]
&= \frac{\det B[\gamma \cup \{p,i\}, \gamma \cup\{p,j\}]}{\det B[\gamma]} \\
&=\frac{\det B[\alpha \cup \{i\}, \alpha \cup\{j\}]}{\det B[\gamma]} \\
& \neq 0.
\end{align*}
Hence, we have a contradiction.
\epf

Now that we have the $\tt NN$ Theorem (for apr-sequences),
a question arises:
Does a statement analogous to
Theorem \ref{qpr: N theorem} hold for apr-sequences?
It does not:

\begin{ex}
\rm
For the matrix
\[
B =
\mtx{0&1&0&0\\
        1&0&0&0\\
        0&0&0&1\\
        0&0&1&0},
\]
$\apr(B) = \tt SNS$.
\end{ex}

The next result is a corollary to the $\tt NN$ Theorem.

\begin{cor}
Let $B\in\Fnn$ be symmetric.
Suppose that $\apr(B)$ contains $\tt NN$.
If $\apr(B) \neq \tt NN\overline{N}$,
then $B$ is singular.
\end{cor}

\bpf
We will establish the contrapositive.
Suppose that $B$ is nonsingular.
Let $\apr(B) = a_1a_2 \cdots a_{n-1}$, and
suppose that $a_k a_{k+1}= \tt NN$ for some $k$.
By the $\tt NN$ Theorem, $a_j = \tt N$ for all $j \geq k$.
Hence, $\apr(B) = a_1a_2 \cdots a_{k-1} \tt NN\overline{N}$.
By the Inverse Theorem,
$\apr(B^{-1}) = {\tt \overline{N} NN}a_{k-1} \cdots  a_2a_1$.
Then, by the $\tt NN$ Theorem, $a_j = \tt N$ for all $j \leq k-1$,
implying that $\apr(B) = \tt NN \overline{N}$, as desired.
\epf

We now show that $\tt NA$ does not occur as a subsequence of the apr-sequence of a symmetric matrix $B \in \Fnn$.
But, to do so, we need a lemma:

\begin{lem}
\label{zero diagonal}
Let $B\in\Fnn$ be symmetric.
Suppose that
$\apr(B) = a_1{\tt N}a_3 \cdots a_{n-1}$ and
$\epr(B) = {\tt N}\ell_2\ell_3 \cdots \ell_n$.
Then $\apr(B)$ does not contain $\tt A$.
\end{lem}

\bpf
If $B=O_n$, then the desired conclusion follows
by noting that $\apr(B) = \tt NN\overline{N}$.
Suppose that $B \neq O_n$, and
let $B = [b_{ij}]$.
By hypothesis, $b_{ii} = 0$ for all $i$,
implying that $B$ must contain a nonzero off-diagonal entry.
Without loss of generality, we may assume that $b_{12} \neq 0$.
Since every order-2 almost-principal minor of $B$ is zero,
$\det B[\{1,2\}, \{1,j\}] = 0$ and
$\det B[\{1,2\}, \{2,j\}] = 0$ for all
$3 \leq j \leq n$.
Then, as
$\det B[\{1,2\}, \{1,j\}] = - b_{1j}b_{21} = - b_{1j}b_{12}$ and
$\det B[\{1,2\}, \{2,j\}] = b_{12}b_{2j}$,
$b_{1j} = b_{2j} = 0$ for all $3 \leq j \leq n$.
Since $B$ is symmetric,
$b_{j1} = b_{j2} = 0$ for all $3 \leq j \leq n$, implying that
$B$ is a block-diagonal matrix with a $2 \times 2$ block.
It follows that $a_1 \neq \tt A$.
Now, note that, for all $k \in \{2, 3, \dots, n-1\}$,
the $k \times k$ almost-principal submatrix
$B[[k], [k+1] \setminus \{2\}]$
(where $[p] := \{1,2, \dots , p\}$) is singular,
since its first row is zero, as $b_{11}=0$.
Hence, $a_{k} \neq \tt A$ for all $k \geq 2$.
\epf

\begin{thm}
\label{NA}
The sequence $\tt NA$ does not occur as a subsequence of the apr-sequence of a symmetric matrix over a field $\F$.
\end{thm}

\bpf
Let $B\in\Fnn$ be symmetric,
$\apr(B) = a_1a_2 \cdots a_{n-1}$ and
$\epr(B) = \ell_1\ell_2 \cdots \ell_{n}$.
Suppose to the contrary that $a_ka_{k+1} = \tt NA$ for some $k$.
By Observation \ref{N...}, $k \geq 2$.
We now show that $\ell_{k-1}= \tt N$.
Suppose to the contrary that $\ell_{k-1} \neq \tt N$.
Let $B[\gamma]$ be nonsingular with $|\gamma| = k-1$.
By the Schur Complement Corollary,
$\apr(B/B[\gamma]) = \tt NA \cdots $,
which contradicts Observation \ref{N...}.
Hence, $\ell_{k-1}= \tt N$.
By Lemma \ref{zero diagonal}, $k-1 \geq 2$.
Since $a_{k+1} = \tt A$, $\rank(B) \geq k+1$.
Then, as the rank of $B$ is principal,
the $\tt NN$ Theorem for epr-Sequences implies that
$\ell_{k-2} \neq \tt N$.
Let $B[\mu]$ be nonsingular with $|\mu| = k-2$.
Since $a_ka_{k+1} = \tt NA$,
the Schur Complement Corollary implies that
$\apr(B/B[\mu]) = \tt YNA \cdots$ for some $\tt Y \in \{\tt A,N,S\}$.
Since $\ell_{k-1} = \tt N$,
the Schur Complement Theorem implies that
$\epr(B/B[\mu]) = \tt N \cdots$, which contradicts
Lemma \ref{zero diagonal}.
\epf

The fact that an analogous version of
Theorem \ref{qpr: N theorem} does not hold in general for apr-sequences
raises a natural question:
What restrictions (if any) can be added to the hypothesis of
Theorem \ref{qpr: N theorem} in order to
have its conclusion hold for apr-sequences?
Requiring the apr-sequence to contain $\tt A$ as a
subsequence is one such restriction:

\begin{thm}\label{...A...}
Let $B \in \Fnn$ be symmetric.
Suppose that $\apr(B) = a_1 a_2  \cdots a_{n-1}$ and
$a_{k} = \tt A$ for some $k$.
Then neither $\tt NA$ nor $\tt NS$ is a subsequence of $\apr(B)$.
Equivalently, if $a_t = \tt N$ for some $t$,
then $a_j = \tt N$ for all $j \geq t$.
\end{thm}

\bpf
By Theorem \ref{NA}, $\apr(B)$ does not contain $\tt NA$.
Suppose to the contrary that
$a_{p}a_{p+1} = \tt NS$ for some $p$.
Obviously, $p \neq k$ and $p \neq k-1$.
Thus,  $p \leq k-2$ or $p \geq k+1$.
Let $\epr(B) = \ell_1\ell_2 \cdots \ell_n$.
We now examine all possibilities in two cases.

\noindent
{\bf Case 1}: $p \leq k-2$.

\noindent
By Observation \ref{N...}, $p \geq 2$.
We now show that $\ell_{p-1} = \tt N$.
If $\ell_{p-1} \neq \tt N$, then $B$ has a nonsingular
$(p-1) \times (p-1)$ principal submatrix, say, $B[\gamma]$,
which would imply that
$\apr(B/B[\gamma]) = \tt N \cdots A \cdots$
(by the Schur Complement Corollary),
contradicting Observation \ref{N...}.
Hence, $\ell_{p-1} = \tt N$.
By Lemma \ref{zero diagonal}, $p \geq 3$.
Since $a_k = \tt A$, $\rank(B) \geq k > p$.
Then, as $\ell_{p-1} = \tt N$ and the rank of $B$ is principal,
$\ell_{p-2} \neq \tt N$ (see the $\tt NN$ Theorem for epr-Sequences).
Let $B[\mu]$ be a nonsingular $(p-2) \times (p-2)$
(principal) submatrix.
Then, by the Schur Complement Corollary and Schur Complement Theorem,
$\apr(B/B[\mu]) = \tt XN \cdots A \cdots$ and
$\epr(B/B[\mu]) = \tt N \cdots$, for some $\tt X \in \{A,N,S\}$,
contradicting Lemma \ref{zero diagonal}.

\noindent
{\bf Case 2}: $p \geq k+1$.

\noindent
Since $a_{p+1} = \tt S$, $\rank(B) \geq p+1$.
We proceed by considering two cases.

\noindent
{\bf Subcase A}:
$B$ contains a nonsingular $(p+1) \times (p+1)$ principal submatrix.

\noindent
Let $B[\alpha]$ be nonsingular with $|\alpha| = p+1$.
By the Inheritance Theorem,
$\apr(B[\alpha]) = \cdots \tt A \cdots N$.
Then, by the Inverse Theorem,
$\apr((B[\alpha])^{-1}) = \tt N \cdots A \cdots$,
which contradicts Observation \ref{N...}.

\noindent
{\bf Subcase B}:
$B$ does not contain a nonsingular $(p+1) \times (p+1)$
principal submatrix.

\noindent
Clearly, $\ell_{p+1} = \tt N$.
Since $\rank(B) \geq p+1$, and because the rank of $B$ is principal,
the $\tt NN$ Theorem for epr-Sequences implies that
$B$ contains a nonsingular $(p+2) \times (p+2)$
principal submatrix, say, $C$.
Let $\apr(C) = a'_1a'_2 \cdots a'_{p+1}$
and $\epr(C) = \ell'_1 \ell'_2 \cdots \ell'_{p+1} \tt A$.
By the Inheritance Theorem,
$a'_k = \tt A$ and $a'_p = \tt N$.
Since every principal submatrix of $C$ is
also a principal submatrix of $B$,
$\ell'_{p+1} = \tt N$.
Thus far, we have that
$\apr(C) = a'_1a'_2 \cdots a'_{k-1}{\tt A} \cdots {\tt N}a'_{p+1}$
and
$\epr(C) = \ell'_1 \ell'_2 \cdots \ell'_{p} \tt NA$.
By the Inverse Theorem and the Inverse Theorem for epr-Sequences,
$\apr(C^{-1}) =
a'_{p+1}{\tt N} \cdots {\tt A}a'_{k-1} \cdots a'_{2}a'_{1}$
and
$\epr(C^{-1}) = {\tt N}\ell'_p \cdots \ell'_2 \ell'_{1} {\tt A}$,
which contradicts Lemma \ref{zero diagonal}.
\epf

%%%%%%%%%%%%%%%%%%%%%%%%%%%%%%%
%%%%%%%%%%%%%%%%%%%%%%%%%%%%%%%
%%%%%%%%%%%%%%%%%%%%%%%%%%%%%%%
%%%%%%%%%%%%%%%%%%%%%%%%%%%%%%%
%%%%%%%%%%%%%%%%%%%%%%%%%%%%%%%
%%%%%%%%%%%%%%%%%%%%%%%%%%%%%%%
%%%%%%%%%%%%%%%%%%%%%%%%%%%%%%%
%%%%%%%%%%%%%%%%%%%%%%%%%%%%%%%
\section{Sequences not containing an $\tt A$}\label{s: No As}
$\null$
\indent
In this section, we confine our attention to
the apr-sequences not containing $\tt A$ as a subsequence,
for which a complete characterization will be provided
(see Theorem \ref{No As}).
This characterization is then used to obtain a necessary condition for
a sequence to be the apr-sequence of a
symmetric matrix over an arbitrary field $\F$.
We will start by focusing on sequences that begin with $\tt SN$.
We introduce two matrices that are central to this section:
\[
L_2(a) := \mtx{1&1\\1&a} \mbox{ \ and \ \ }
A(K_2) := \mtx{0&1\\1&0},
\]
where $a \in \F$.

\begin{lem}\label{SN lemma}
Let $B \in \Fnn$ and $n \geq 3$.
Suppose that $\apr(B) = a_1 a_2  \cdots a_{n-1}$.
Then the following statements hold:
\ben
\item
If $B = J_{n-k} \oplus O	_k$
for some integer $k$ with $1 \leq k \leq n-1$, then
$\apr(B) = \tt SN\overline{N}$.
\item
If $B = L_2(a) \oplus O	_{n-2}$ for some $a \in \F$, then
$\apr(B) = \tt SN\overline{N}$.

\item
If  $B=A(K_2) \oplus A(K_2) \oplus \cdots \oplus A(K_2) \oplus O	_k$
for some integer $k$ with $0 \leq k \leq n-2$,
then $\apr(B) = \tt S\overline{NS} \hspace{0.4mm} \overline{N}$, with
$\tt \overline{N}$ containing $k$ copies of $\tt N$.
\een
\end{lem}

\bpf
The verification of Statements (1) and (2) is omitted, 
since it is trivial.
%%%To establish Statement (1),
%%%let $k$ be an integer with $1 \leq k \leq n-1$, and
%%%suppose that $B = J_{n-k} \oplus O	_k$.
%%%Obviously, $a_1 = \tt S$.
%%%Since $\rank(B) = 1$,
%%%every square submatrix of $B$ order 2 or larger is singular.
%%%Hence, $\apr(B) = \tt SN\overline{N}$.
%%%
%%%
%%%To establish Statement (2),
%%%let $a \in \F$, and
%%%suppose that $B = L_2(a) \oplus O	_{n-2}$.
%%%Obviously, $a_1 = \tt S$.
%%%Since $\rank(B) \leq 2$,
%%%every square submatrix of $B$ order 3 or larger is singular.
%%%Hence, $\apr(B) = {\tt S} a_2\overline{\tt N}$.
%%%Finally, to see that $a_2=\tt N$,
%%%notice that
%%%every $2 \times 2$ almost-principal submatrix of $B$ contains
%%%a zero row or a zero column.
Statement (3) is established by examining two cases.
First, consider the case when
$B=A(K_2) \oplus O_k$ with $k \geq 1$.
In that case, obviously, $a_1 = \tt S$, and,
since every almost-principal submatrix of $B$ of order 2 or
larger would contain a zero row or a zero column,
we would have $\apr(B) = {\tt SN} \overline{\tt N}$.

Finally,
to establish the remaining cases of Statements (3),
by Observation \ref{append},
it suffices to show that the matrix
\[
C=A(K_2) \oplus A(K_2) \oplus \cdots \oplus A(K_2),
\]
with at least two copies of $A(K_2)$,
has apr-sequence $\tt SNS\overline{NS}$.
%since appending a zero row and a zero column to matrix whose 
%apr-sequence does not contain an $\tt A$ results in a matrix whose 
%apr-sequence is that of the original matrix with an 
%$\tt N$ appended at the end.
Let $\apr(C) = a'_1a'_2 \cdots a'_{m-1}$,
where $m \geq 4$  (thus, $C$ is an $m\times m$ matrix and $m$ is even).
Clearly, $a'_1 = \tt S$.
Let $p$ be an odd integer with $3 \leq p \leq m-1$.
We now show that $a'_p = \tt S$.
Notice that
\[
B[\{1,2, \dots, p\}, \{1,2, \dots, p+1\}\setminus \{p\}] =
A(K_2) \oplus A(K_2) \oplus \cdots \oplus A(K_2) \oplus J_1
\]
(with $\frac{p-1}{2}$ copies of $A(K_2)$), which is nonsingular,
and that
$B[\{1,2, \dots, p\}, \{2, 3,\dots, p+1\}]$  is singular,
since its second row consists entirely of zeros (i.e., it is a zero row).
Hence, $a'_p = \tt S$, as desired.

We now show that $a'_q = \tt N$ if $q$ is even.
First, observe that any principal submatrix of $B$ of
odd order contains a zero row and a zero column.
Let $q$ be an even integer with $2 \leq q \leq m-2$,
and suppose that
$B[\alpha \cup \{i\}, \alpha \cup \{j\}]$ is a $q \times q$
almost-principal submatrix;
thus, $i \neq j$ and $|\alpha| = q-1$.
Hence,
$B[\alpha \cup \{i,j\}]$ is a $(q+1) \times (q+1)$
(principal) submatrix of odd order,
implying that  $B[\alpha \cup \{i,j\}]$ contains
a zero row and a zero column.
Hence, any $q \times q$ almost-principal submatrix of
$B[\alpha \cup \{i,j\}]$ contains either
a zero row or a zero column.
Then, as
$B[\alpha \cup \{i\}, \alpha \cup \{j\}]$ is a submatrix of
$B[\alpha \cup \{i,j\}]$,
$B[\alpha \cup \{i\}, \alpha \cup \{j\}]$ is singular.
It follows that $a'_q= \tt N$.
We conclude that
$\apr(C)=\tt SNS\overline{NS}$, as desired.
\epf

\begin{prop}\label{SN}
Let $B \in \Fnn$ be symmetric.
Suppose that $\apr(B) = {\tt SN}a_3 a_4 a_5 \cdots a_{n-1}$.
Then one of the following statements holds:
\ben
\item
There exists a generalized permutation matrix $P$ and
a nonzero constant $c \in \F$ such that
$cP^{T}BP = J_{n-k} \oplus O	_k$
for some integer $k$ with $1 \leq k \leq n-1$.
Moreover, $\apr(B) = \tt SN\overline{N}$.
\item
There exists a generalized permutation matrix $P$,
a nonzero constant $c \in \F$ and $a \in \F$ such that
$cP^{T}BP = L_2(a) \oplus O	_{n-2}$.
Moreover, $\apr(B) = \tt SN\overline{N}$.

\item
There exists a generalized permutation matrix $P$ such that
$P^TBP=
A(K_2) \oplus A(K_2) \oplus \cdots \oplus A(K_2) \oplus O	_k$
for some integer $k$ with $0 \leq k \leq n-2$.
Moreover, $\apr(B) = \tt S\overline{NS} \hspace{0.4mm} \overline{N}$,
with $\tt \overline{N}$ containing $k$ copies of $\tt N$.
\een
\end{prop}

\bpf
Suppose that $B=[b_{ij}]$.
Since $\apr(B)$ begins with $\tt S$,
$B$ contains at least one nonzero off-diagonal entry.
It suffices to show that the first sentence of one of
Statements (1), (2) and (3) holds,
since the remaining part of the statements follows immediately from 
Lemma \ref{SN lemma}.
We proceed by examining two cases.

\noindent
{\bf Case 1}: $B$ contains a row with
more than one nonzero off-diagonal entry.

\noindent
We now show that Statement (1) holds.
Since a simultaneous permutation of
the rows and columns of $B$ leaves $\apr(B)$ invariant,
we may assume that
the first row of $B$ contains more than one nonzero off-diagonal entry.
Furthermore, we may assume that
$b_{1j} \neq 0$ for $j \in \{2,3, \dots,p\}$
for some $p \geq 3$,
and that
$b_{1j} = b_{j1} = 0$ for $j \in \{p+1, p+2, \dots,n\}$
(note that $\{p+1, p+2, \dots,n\}$ is empty if $n=3$).
Let
$\alpha = \{2, 3, \dots, p\}$ and
$\beta = \{p+1, p+2, \dots n\}$.
Since $a_2 = \tt N$,
$b_{11}b_{23} - b_{12}b_{13} = \det B[\{1,2\}, \{1,3\}] = 0$.
Then, as $b_{12}b_{13} \neq 0$, $b_{11} \neq 0$.
Since multiplying $B$ by a nonzero constant
leaves $\apr(B)$ invariant, we may assume that $b_{11} = 1$.
Furthermore, we may assume that
$b_{1j} = b_{j1} = 1$ for all $j \in \alpha$,
as multiplying a row and column of $B$ by
a nonzero constant leaves $\apr(B)$ invariant.

Thus far, we have that
$b_{11}=1$, that
$b_{1j} = b_{j1} = 1$ for all $j \in \alpha$, and that
$b_{1j} = b_{j1} = 0$ for all $j \in \beta$.
Now, note that
for all $i,j \in\{1,2, \dots, n\}$, with $i \neq j$,
and all $t \in\{1,2, \dots, n\} \setminus \{i,j\}$,
\[
\det B[\{t,i\}, \{t,j\}] = b_{tt}b_{ij} - b_{ti}b_{tj}.
\]
Since $a_2 = \tt N$,
we have that, for $i \neq j$ and $t \in\{1,2, \dots, n\} \setminus \{i,j\}$,
\[
b_{tt}b_{ij} = b_{ti}b_{tj}.
\]
Since $b_{11} = 1$,
$b_{ij} = b_{1i}b_{1j}$ for all
$i,j \in \{2,3, \dots, n\}$ with $i \neq j$.
Thus, if $i,j \in \alpha$ and $i \neq j$, then $b_{ij}=1$,
implying that every off-diagonal entry of $B[\alpha]$ is 1.
Moreover, if
$i \in \beta$ or
$j \in \beta$,
with $i \neq j$,
then $b_{ij}=0$.
Hence, $B[\beta]$ is a diagonal matrix and
$B=B[\alpha] \oplus B[\beta]$.

We now show that
$B[\alpha] = J_p$.
Observe that if $t \in \alpha$ and
$i,j \in \{1,2, \dots, p\} \setminus\{t\}$ with $i \neq j$,
then $b_{tt} = b_{ti}b_{tj}/b_{ij} =1$.
Hence, $B[\alpha] = J_p$.

We now show that
$B[\beta]=O_{n-p}$.
Since $a_1 = \tt S$,
it follows that $n>p$,
as otherwise we would have $n=p$,
which would imply
that $B=B[\alpha] = J_p$,
whose apr-sequence is $\tt AN\overline{N}$,
which is a contradiction.
It follows that $\beta$ is nonempty.
If $t \in \beta$, then
$b_{tt} =b_{t1}b_{t2}/b_{12} = 0$.
Thus, $B[\beta]=O_{n-p}$.
Then, with $k:=n-p$, we have that
$B=J_{n-k} \oplus O	_k$.
It is easy to see that the operations performed on $B$ that
resulted in the matrix $J_{n-k} \oplus O	_k$ is accomplished by
finding an appropriate generalized permutation matrix $P$ and
a nonzero constant $c$ such that
$cP^{T}BP = J_{n-k} \oplus O	_k$.
Moreover, observe that
$1 \leq k \leq n-3 \leq n-1$.
Hence, Statement (1) holds.

\noindent
{\bf Case 2}: 
%$B$ does not contain a row with
%more than one nonzero off-diagonal entry.
Each row of $B$ contains at most one nonzero off-diagonal entry.

\noindent
Since a simultaneous permutation of
the rows and columns of $B$ leaves $\apr(B)$ invariant,
we may assume that
$b_{12}\neq 0$  and $b_{1j} =0$ for  $j \in \{2,3, \dots, n\}$.
Since $B$ does not contain a row with
more than one nonzero off-diagonal entry, and because $B$ is symmetric,
$B=B[\{1,2\}] \oplus B[\{3,4, \dots, n\}]$.
Moreover, since multiplying a row and column of $B$ by
a nonzero constant leaves $\apr(B)$ invariant,
we may assume that $b_{12}=b_{21}=1$.
Then, as $a_2= \tt N$,
$0=\det B[\{1,j\},\{2,j\}] = b_{jj}$
for $j \in \{3,4, \dots, n\}$.
Hence,
$B[\{3,4, \dots, n\}]$ has zero diagonal.

\noindent
{\bf Subcase A}: $B[\{3,4, \dots, n\}] = O_{n-2}$.

\noindent
If $b_{11}=b_{22}=0$, then
$B=A(K_2) \oplus O_{n-2}$,
implying that Statement (3) holds.
Now, suppose that
$b_{11} \neq 0$ or $b_{22} \neq 0$.
We may assume that $b_{11} \neq 0$.
Then by multiplying $B$ by $\frac{1}{b_{11}}$, and
then multiplying row 2 and column 2 of $B$ by $b_{11}$,
we obtain the matrix $L_2(a) \oplus O_{n-2}$ for some $a$.
Without loss of generality, we may assume that
$B=L_2(a) \oplus O_{n-2}$.
It is easy to verify that the operations performed on $B$ that
led to the matrix $L_2(a) \oplus O_{n-2}$ is accomplished by
finding an appropriate generalized permutation matrix $P$ and
a nonzero constant $c$ such that
$cP^{T}BP = L_2(a) \oplus O_{n-2}$.
Hence, Statement (2) holds.

\noindent
{\bf Subcase B}: $B[\{3,4, \dots, n\}] \neq O_{n-2}$.

\noindent
Since $B[\{3,4, \dots, n\}]$ has zero diagonal,
$B[\{3,4, \dots, n\}]$ must have a nonzero off-diagonal entry.
Without loss of generality,
we may assume that $b_{34} \neq 0$ and
$b_{3j} = 0$ for $j \in \{4,5, \dots n\}$.
Then, as $a_2 = \tt N$,
$0=\det B[\{1,3\},\{1,4\}]=b_{11}b_{34}$
and  $0=\det B[\{2,3\},\{2,4\}]=b_{22}b_{34}$.
Since $b_{34} \neq 0$, $b_{11} =b_{22}= 0$.
Hence, $B=A(K_2) \oplus B[\{3,4, \dots, n\}]$.
Since every almost-principal minor of
$B[\{3,4, \dots, n\}]$ is an almost-principal minor of $B$,
$\apr(B[\{3,4, \dots, n\}])$ begins with $\tt SN$.
By our assumption in the present case (Case 2),
$B[\{3,4, \dots, n\}]$ does not contain a row with
more than one nonzero off-diagonal entry.
Thus, we can apply our findings in
Subcase A and Subcase B of the present case (Case 2) to
the matrix $B[\{3,4, \dots, n\}]$:
Since $B[\{3,4, \dots, n\}]$ has zero diagonal,
we conclude that we may assume that either
$B[\{3,4, \dots, n\}] = A(K_2) \oplus B[\{5,6, \dots, n\}]$
(if $n \geq 5$) or
$B[\{3,4, \dots, n\}] = A(K_2)$ (if $n=4$).
Hence, we must have
$B=A(K_2) \oplus A(K_2)$ if $n=4$, and
$B = A(K_2) \oplus A(K_2) \oplus B[\{5,6, \dots, n\}]$ if $n \geq 5$.
It is not hard to see now that
continuing this process will allow us to assume,
without loss of generality, that
$B=A(K_2) \oplus A(K_2) \oplus \cdots \oplus A(K_2) \oplus O	_k$
for some integer $k$ with $0 \leq k \leq n-2$
(where the parity of $k$ is the same as that of $n$);
hence, Statement (3) holds.
\epf

We now turn our attention to sequences that begin with $\tt SS$.

\begin{prop}\label{SS...NS}
A sequence of the form
${\tt SS}a_3a_4 \cdots a_{n-3}{\tt NS}$
is not the apr-sequence of a symmetric matrix over $\F$.
\end{prop}

\bpf
Suppose to the contrary that there exists
a symmetric matrix $B \in \Fnn$ such that
$\apr(B) = {\tt SS}a_3a_4 \cdots a_{n-3}{\tt NS}$.
Observe that $B$ is singular
(otherwise, the Inverse Theorem would imply that $B^{-1}$ has
apr-sequence ${\tt SN} a_{n-3} \cdots a_4a_3 \tt SS$, 
which would contradict Proposition \ref{SN}).
Hence, $\rank(B) \leq n-1$.
Since $\apr(B)$ ends with $\tt S$,
$B$ contains a nonsingular $(n-1) \times (n-1)$
(almost-principal) submatrix,
implying that $\rank(B) = n-1$.
Since the rank of $B$ is principal,
$B$ contains a nonsingular $(n-1) \times (n-1)$
principal submatrix, say, $B'$.
Without loss of generality, we may assume that
$B'=B[\{1,2, \dots, n-1\}]$.
By the Inheritance Theorem,
$\apr(B')$ ends with $\tt N$.
Then, as $B'$ is nonsingular,
$(B')^{-1}$ is a diagonal matrix,
implying that $B'$ is also a diagonal matrix.
Then, as $B'$ is nonsingular,
$b_{jj} \neq 0$ for all $j \in \{1,2, \dots, n-1\}$.
Since $\apr(B)$ begins with $\tt S$,
the last row (and last column) of $B$ must
contain a nonzero off-diagonal entry.
Without loss of generality,
we may assume that $b_{1n} \neq 0$.
Now, note that
\[
\det(B[\{1,2, \dots, n-2\},\{1,2, \dots, n\} \setminus \{1,n-1\}]) =
(-1)^{n-1}b_{1n} \prod_{j=2}^{n-2}b_{jj} \neq 0.
\]
Hence, $B$ contains an $(n-2) \times (n-2)$
nonsingular almost-principal submatrix,
which contradicts the fact that
$\apr(B) = {\tt SS}a_3a_4 \cdots a_{n-3}{\tt NS}$.
\epf

\begin{lem}\label{NS lemma}
Let $B \in \Fnn$ be symmetric, and
let $k$ be an even integer.
Suppose that $\apr(B) = a_1 a_2 \cdots a_{n-1}$ and
$a_{k}a_{k+1} = \tt NS$.
Let $B'$ be a $(k+2) \times (k+2)$ principal submatrix of $B$.
If $B' = A(K_2) \oplus A(K_2) \oplus \cdots \oplus A(K_2)$,
then $a_1 a_2 = \tt SN$.
\end{lem}

\bpf
Suppose that
$B' = A(K_2) \oplus A(K_2) \oplus \cdots \oplus A(K_2)$.
Thus, $a_1 = \tt S$.
If $k=2$, then there is nothing to prove, and
if $n=k+2$, then the desired conclusion follows from
Lemma \ref{SN lemma};
thus, we assume that
$k \geq 4$, and that
$n \geq k+3$.
Without loss of generality, we may assume that
$B' = B[\{1,2, \dots, k+2\}]$.
We now show that
$B[\{k+3,k+4, \dots,n\}]$ has zero diagonal,
and that
\[
B = B' \oplus B[\{k+3,k+4, \dots,n\}].
\]
Suppose that $B=[b_{ij}]$.
To see that $B[\{k+3,k+4, \dots,n\}]$ has zero diagonal,
let $j \in \{k+3, k+4, \dots,n\}$ and
$\alpha =\{1,2,\dots, k\} \cup\{j\}$.
Then, as $a_k=\tt N$,
\[
0=
\det B[\alpha \setminus \{1\}, \alpha  \setminus \{2\}]=
(-1)^{\frac{k-2}{2}}b_{jj}\det(M),
\]
where $M$ is  the $(k-1) \times (k-1)$ matrix
\[
M=J_1 \oplus A(K_2) \oplus A(K_2) \oplus \cdots \oplus A(K_2).
\]
Then, as $\det(M) \neq 0$,
$b_{jj} = 0$.
It follows that
$B[\{k+3,k+4, \dots,n\}]$ has zero diagonal, as desired.

Now, to show that
\[
B = B' \oplus B[\{k+3,k+4, \dots,n\}],
\]
let
$p,q \in \{1,2, \dots, k+2\}$ and
$j \in \{k+3, k+4, \dots,n\}$,
where $p$ is odd and $q$ is even.
We now show that
$b_{pj}=0$ and $b_{qj}=0$.

\noindent
{\bf Case 1}: $p \leq k-2$ and $q \leq k-2$.

\noindent
Let
$\alpha =\{1,2,\dots, k\}$.
It is easy to see that
\[
\det B[\alpha, (\alpha  \setminus \{p+1\}) \cup \{j\}]=
(-1)^{p+k} b_{pj}\det(F)
\]
and
\[
\det B[\alpha, (\alpha  \setminus \{q-1\}) \cup \{j\}]=
(-1)^{q+k} b_{qj}\det(G),
\]
for some matrices
$F$ and $G$ that are permutationally similar to the
$(k-1) \times (k-1)$ matrix
\[
J_1 \oplus A(K_2) \oplus A(K_2) \oplus \cdots \oplus A(K_2).
\]
Since $a_k = \tt N$,
$\det B[\alpha, (\alpha  \setminus \{p+1\}) \cup \{j\}]=0$ and
$\det B[\alpha, (\alpha  \setminus \{q-1\}) \cup \{j\}]=0$.
Then, as $F$ and $G$ are nonsingular,
$b_{pj}=0$ and $b_{qj}=0$, as desired.

\noindent
{\bf Case 2}: $p > k-2$ and $q > k-2$.

\noindent
Let
$\beta =\{1,2,\dots, k-2\}$.
It is easy to see that
\[
\det B[\beta \cup \{p, p+1\}, \beta \cup \{p,j\}]=
-b_{pj}\det(H)
\]
and
\[
\det B[\beta \cup \{q-1,q\}, \beta \cup \{q,j\}]=
b_{qj}\det(H),
\]
where $H$ is the $(k-1) \times (k-1)$ matrix
\[
H=A(K_2) \oplus A(K_2) \oplus \cdots \oplus A(K_2) \oplus J_1.
\]
Since $a_k = \tt N$,
$\det B[\beta \cup \{p, p+1\}, \beta \cup \{p,j\}] =0$ and
$\det B[\beta \cup \{q-1,q\}, \beta \cup \{q,j\}]=0$.
Then, as $H$ is nonsingular,
$b_{pj}=0$ and $b_{qj}=0$, as desired.

It follows from Case 1 and Case 2 that
\[
B = B' \oplus B[\{k+3,k+4, \dots,n\}],
\]
as desired.

To conclude the proof, we show that $a_2 = \tt N$.
If $B[\{k+3,k+4, \dots,n\}]$ is the zero matrix,
then the desired conclusion follows from Lemma \ref{SN lemma}.
Thus, assume that  $B[\{k+3,k+4, \dots,n\}]$ is not the zero matrix.
Since $B' = A(K_2) \oplus A(K_2) \oplus \cdots \oplus A(K_2)$,
it suffices to show that
there exists a generalized permutation matrix $P$ such that
\[
P^TB[\{k+3,k+4, \dots,n\}]P=
A(K_2) \oplus A(K_2) \oplus \cdots \oplus A(K_2) \oplus O	_t
\]
for some integer $t$ with $t \geq 0$,
since that would imply that
there exists a generalized permutation matrix $Q$ such that
\[
Q^TBQ=
A(K_2) \oplus A(K_2) \oplus \cdots \oplus A(K_2) \oplus O	_t,
\]
and, therefore,
that $a_2 = \tt N$ (see Lemma \ref{SN lemma}).
Since $B[\{k+3,k+4, \dots,n\}]$ is a
nonzero matrix with zero diagonal,
$n \geq k+4$.
If $n=k+4$,
then the fact that
$B[\{k+3,k+4, \dots,n\}]$ has zero diagonal immediately implies that
there exists a generalized permutation matrix $P$ such that
$P^TB[\{k+3,k+4, \dots,n\}]P = A(K_2)$, as desired.
Thus, we assume that $n \geq k+5$
(implying that the order of $B[\{k+3,k+4, \dots,n\}]$ is
greater than or equal to $3$).
Since $B[\{k+3,k+4, \dots,n\}]$ has zero diagonal,
Proposition \ref{SN} implies that it suffices to show that
$\apr(B[\{k+3,k+4, \dots,n\}])$ begins with $\tt SN$.
We start by showing that
$\apr(B[\{k+3,k+4, \dots,n\}])$ begins with $\tt S$.
Since $B[\{k+3,k+4, \dots,n\}]$ is a nonzero matrix with zero diagonal,
$\apr(B[\{k+3,k+4, \dots,n\}])$ does not begin with $\tt N$.
Suppose to the contrary that
$\apr(B[\{k+3,k+4, \dots,n\}])$ begins with $\tt A$.
Since $B[\{k+3,k+4, \dots,n\}]$ has zero diagonal,
$B[\{k+3,k+4, \dots,n\}] = J_{n-k-2} - I_{n-k-2}$.
Let $\theta = \{1,2, \dots, k-2\}$.
Note that
$B[\theta] = A(K_2) \oplus A(K_2) \oplus \cdots \oplus A(K_2)$,
and that
\[
B[\theta \cup \{k+3,k+4\}, \theta \cup \{k+3, k+5\}] =
B[\theta] \oplus B[\{k+3,k+4\}, \{k+3, k+5\}].
\]
Then, as
\[
B[\{k+3,k+4\}, \{k+3, k+5\} =
\mtx{0&1 \\
        1&1}
\]
is nonsingular,
$B[\theta] \oplus B[\{k+3,k+4\}, \{k+3, k+5\}]$ is nonsingular,
implying that
$B[\theta \cup \{k+3,k+4\}, \theta \cup \{k+3, k+5\}]$ is nonsingular,
a contradiction to the fact that $a_k = \tt N$.
Hence, it follows that
$\apr(B[\{k+3,k+4, \dots,n\}])$ begins with $\tt S$.

It now remains to show that the second letter in
$\apr(B[\{k+3,k+4, \dots,n\}])$ is $\tt N$.
Let $p,q,r \in \{k+3, k+4, \dots, n\}$ be distinct integers, and
let $\mu = \{1,2, \dots, k+2\}$.
We now show that $\det(B[\{p,q\},\{p,r\}])=0$.
Notice that $B[\mu]=B'$, and that
\[
B[\mu \cup \{p,q\}, \mu \cup\{p,r\}] =
B' \oplus B[\{p,q\},\{p,r\}].
\]
Since $a_k = \tt N$,
$B[\mu \cup \{p,q\}, \mu \cup\{p,r\}]$
is singular.
Then, as $B'$ is nonsingular,
$B[\{p,q\},\{p,r\}]$ is singular,
implying that $\det(B[\{p,q\},\{p,r\}])=0$, as desired.
Hence, every $2 \times 2$ almost-principal submatrix of
$B[\{k+3,k+4, \dots,n\}]$ is singular,
implying that
$\apr(B[\{k+3,k+4, \dots,n\}])$ begins with $\tt SN$,
as desired.
\epf

\begin{thm}\label{NS}
Let $B \in \Fnn$ be symmetric.
Suppose that $\apr(B) = a_1 a_2 \cdots a_{n-1}$ and
$a_{k}a_{k+1} = \tt NS$ for some $k$.
Then $a_1 a_2 = \tt SN$.
\end{thm}

\bpf
Suppose to the contrary that $a_1 a_2 \neq \tt SN$.
Since $\apr(B)$ contains $\tt NS$,
Observation \ref{N...} implies that $a_1 \neq \tt N$, and
Theorem \ref{...A...} implies that $a_1a_2$ does not contain an $\tt A$.
It follows that $a_1a_2 = \tt SS$.
Since $a_{k+1} = \tt S$, $B$ contains a nonsingular
$(k+1) \times (k+1)$ almost-principal submatrix, say,
$B[\alpha \cup \{i\}, \alpha \cup \{j\}]$
(thus, $i \neq j$ and $|\alpha| = k$).
Let $B' = B[\alpha \cup \{i,j\}]$ and
$\apr(B') = a'_1a'_2 \cdots a'_{k+1}$.
Since $a_k = \tt N$,
the Inheritance Theorem implies that  $a'_k= \tt N$.
Since $B'$ contains the nonsingular
$(k+1) \times (k+1)$ almost-principal submatrix
$B[\alpha \cup \{i\}, \alpha \cup \{j\}]$ as a submatrix,
$a'_{k+1} \in \{\tt A, \tt S\}$.
Thus, $a'_{k}a'_{k+1} \in \{\tt NA, NS \}$.
By Theorem \ref{NA},  $a'_{k}a'_{k+1} = \tt NS$.
Thus, $\apr(B') = a'_1a'_2 \cdots a'_{k-1} \tt NS$.
Since $\apr(B')$ contains $\tt NS$,
Observation \ref{N...} implies that $a'_1 \neq \tt N$,
Theorem \ref{...A...} implies that $a'_1a'_2$ does not contain an $\tt A$,
and Proposition \ref{SS...NS} implies that $a'_1a'_2 \neq \tt SS$.
Thus, $a'_1a'_2 = \tt SN$, and, therefore,
$\apr(B') = {\tt SN}a'_3a'_4 \cdots a'_{k-1} \tt NS$.
It follows from Proposition \ref{SN} that
there exists a generalized permutation matrix $P$ such that
\[
P^TB'P=A(K_2) \oplus A(K_2) \oplus \cdots \oplus A(K_2),
\]
and that $k$ is even.
Without loss of generality, we may assume that
\[
B' = A(K_2) \oplus A(K_2) \oplus \cdots \oplus A(K_2).
\]
By Lemma \ref{NS lemma},
$a_1 a_2 = \tt SN$,
a contradiction to the fact that $a_1a_2 = \tt SS$.
\epf

The sequences not containing an $\tt A$ that are realized as the apr-sequence of a symmetric matrix over a field $\F$ are characterized:

\begin{thm}\label{No As}
Let $a_1 a_2 \cdots a_{n-1}$  be a sequence from
$\{\tt S, N\}$ and $\F$ be a field.
Then
$a_1 a_2 \cdots a_{n-1}$ is the apr-sequence of a
symmetric matrix $B \in \Fnn$
if and only if $a_1 a_2 \cdots a_{n-1}$  has one of the following forms:
\ben
\item
$\tt N\overline{N}$.
\item
$\tt SN\overline{N}$.
\item
$\tt SNS\overline{NS} \hspace{0.4mm} \overline{N}$.
\item
$\tt SS\overline{S} \hspace{0.4mm} \overline{N}$.
\een
\end{thm}

\bpf
Let $\sigma = a_1 a_2 \cdots a_{n-1}$.
Suppose that $\sigma$ is the apr-sequence of
some symmetric matrix $B \in \Fnn$.
If $a_1 = \tt N$, then $\sigma$ has
form (1) (see Observation \ref{N...}).
Thus, assume that $a_1 = \tt S$.
Since the apr-sequence $\tt S$ is
not attainable (by a $2 \times 2$ matrix),
$n \geq 3$.
If $a_2= \tt N$, then Proposition \ref{SN} implies that $\sigma$
must have one of the forms (2) or (3).
Finally, assume that $a_2 = \tt S$.
Then, by  Theorem \ref{NS},
$\sigma$ does not contain $\tt NS$.
It follows that $\sigma$ must have form (4).

For the other direction,
suppose that $\sigma$ has one of the forms (1)--(4).
If $\sigma = \tt N\overline{N}$,
then $\apr(O_{n}) = \sigma$.
If $\sigma$ has the form (2) or (3),
then the desired conclusion follows from
Lemma \ref{SN lemma}.
Finally, suppose that
$\sigma$ has the form (4);
thus, $n \geq 3$.
Because of Observation \ref{append},
it suffices to reach the desired conclusion in the case when
$\sigma$ does not contain an $\tt N$;
thus, we assume that
$\sigma = \tt SS \overline{S}$.
Let $B' = J_2 \oplus I_{n-2}$ and
$\apr(B') = a'_1a'_2 \dots a'_{n-1}$.
We now show that $\apr(B') = \sigma$.
It is obvious that $a'_1 = \tt S$.
Let $k \in \{2,3, \dots, n-1\}$ and
$\alpha = \{1,2, \dots, k+1\}$.
Now, observe that
$C[\alpha \setminus \{1\}, \alpha \setminus \{k+1\}]$ has
a zero row, and that
$C[\alpha \setminus \{1\}, \alpha \setminus \{2\}] = I_{k}$.
Hence, $C$ contains both a
singular and a nonsingular $k \times k$ almost-principal submatrix,
implying that
$\apr(C) = \tt SS\overline{S} = \sigma$,
as desired.
\epf

Combining Theorem \ref{No As} with Theorem \ref{...A...} leads to
a necessary condition for a sequence to be the
apr-sequence of a symmetric matrix over an arbitrary field $\F$:

\begin{thm}\label{necessary condition}
Let $\F$ be a field.
Let $n\geq 3$ and $\sigma = a_1a_2 \cdots a_{n-1}$ be
a sequence from $\{\tt A,N,S\}$.
If $\sigma$ is the apr-sequence of a symmetric matrix $B \in \Fnn$,
then one of the following statements holds:
\ben

\item
$\sigma = \tt SNS\overline{NS} \hspace{0.4mm} \overline{N}$.

\item
Neither $\tt NA$ nor $\tt NS$ is a subsequence of $\sigma$.

\een
\end{thm}

If $\F$ is the field of order $2$, then 
the converse of Theorem \ref{necessary condition} does not hold: 
An exhaustive inspection reveals that
the only sequences starting with $\tt A$ that
are realized as the apr-sequence of a
$4 \times 4$ symmetric matrix over the field of order $2$ are
$\tt AAA$,
$\tt ASS$,
$\tt ASN$ and
$\tt ANN$
(since a simultaneous permutation of
the rows and columns of a matrix leaves its apr-sequence invariant,
this inspection is reduced to checking a total of five matrices);
as this list does not include the sequences
$\tt AAN$,
$\tt AAS$ and
$\tt ASA$,
the converse of Theorem \ref{necessary condition} does not hold if 
$\F$ is the field of order $2$.

%We conjecture that the converse of 
%Theorem \ref{necessary condition} holds  if $\F$ is of characteristic $0$.
%
%
%\begin{conj}\label{conjecture}
%Let $\F$ be a field of characteristic $0$.
%A sequence $\sigma$ from $\{\tt A,N,S\}$ of length $n-1$, with $n\geq 3$,
%is the apr-sequence of a symmetric matrix $B \in \Fnn$
%if and only if
%one of the following statements holds:
%\ben
%\item
%$\sigma = \tt SNS\overline{NS} \hspace{0.4mm} \overline{N}$.
%\item
%Neither $\tt NA$ nor $\tt NS$ is a subsequence of $\sigma$.
%\een
%\end{conj}

For fields of characteristic $0$, 
we suspect that the converse of Theorem \ref{necessary condition} holds.
Since Lemma  \ref{SN lemma} implies that  
$\sigma = \tt SNS\overline{NS} \hspace{0.4mm} \overline{N}$ is realized as the apr-sequence of a symmetric matrix (over any field),
the converse of Theorem \ref{necessary condition} holds if the following statement holds:
If $\F$ is a field  and $\sigma$ is a sequence from 
$\{\tt A,N,S\}$ of length $n-1$, 
with $n\geq 3$, that contains neither 
$\tt NA$ nor $\tt NS$ as a subsequence, 
then $\sigma$ is the apr-sequence of a symmetric matrix in $\Fnn$;
if $\F$ is a field of characteristic $0$,
this statement is reminiscent of, and closely related to, 
Theorem \ref{qpr-char 0},
which is partly why we speculate that 
the converse of Theorem \ref{necessary condition} may hold 
if $\F$ is a field of characteristic $0$.
To establish this statement for fields of characteristic $0$, 
it would be natural to resort to probabilistic techniques akin to 
those used in both \cite{EPR} and \cite{qpr}.
In some instances, these techniques consist of taking an 
$(n-1) \times (n-1)$ symmetric matrix $B$ with 
$\apr(B)=a_1a_2 \cdots a_{n-2}$ and
verifying the existence of a bordering strategy to produce an 
$n \times n$ symmetric matrix $B'$ such that
$\apr(B'):=
a'_1a'_2\cdots a'_{n-1} =
a_1a_2 \cdots a_{n-2}a'_{n-1}$,
where $a'_{n-1}$ is prescribed.
Roughly speaking, the existence of such a $B'$ depends upon sufficient available choice of possible vectors in $\F^{n-1}$.
If $B$ is an \textit{arbitrary} $(n-1) \times (n-1)$ symmetric matrix with 
$\epr(B)  = \tt AA \cdots A$ 
(i.e., the sequence each of whose terms is equal to $\tt A$),
then applying the aforementioned probabilistic techniques to $B$ yield an 
$n \times n$ (symmetric) matrix $B'$ with 
$\epr(B')=\tt AA \cdots AN$ 
(see \cite[Proposition 4.1]{EPR}). 
Similarly, if $C$ is an \textit{arbitrary} 
$(n-1) \times (n-1)$ symmetric matrix with 
$\qpr(C)  = \tt AA \cdots A$,
then applying the aforementioned probabilistic techniques to $C$ yield an 
$n \times n$ (symmetric) matrix $C'$ with 
$\qpr(C')=\tt AA \cdots AN$ 
(see \cite[Lemma 3.1]{qpr}).
The next example shows that 
something similar does not hold for apr-sequences, i.e., that
if $B$ is an \textit{arbitrary} 
$(n-1) \times (n-1)$ symmetric matrix with 
$\apr(B)  = \tt AA \cdots A$,
then applying the aforementioned probabilistic techniques to 
$B$ need not yield an $n \times n$ (symmetric) matrix $B'$ with 
$\apr(B')=\tt AA \cdots AN$.

\begin{ex}\label{Bordering Example 1}\normalfont
Let
\[
B=
\mtx{
-1 & 1 &  1\\
1 & -1 &  1\\
1 &  1  & -1} \in \R^{3\times3}
\mbox{\quad and \quad}
B'=
\mtx{
     B     & \vec{y} \\
\vec{y} &    t} \in \R^{4\times4},
\]
where $\vec{y}$ and $t$ are arbitrary.
Suppose that $\vec{y}=[y_1, y_2, y_3]^T$.
Observe that $\apr(B)=\tt AA$.
We now show that there is no $\vec{y} \in \R^3$ such that
$\apr(B')=\tt AAN$.
Observe that
$\det(B'[\{1,2,3\},\{2,3,4\}])=
2(y_2+y_3)=
-2\det(B'[\{2,3\},\{2,4\}])$.
It follows that if all of the 
order-$3$ almost-principal minors of $B'$ are zero, then 
some order-$2$ almost-principal minor of $B'$ is zero.
Thus,  there is no $\vec{y} \in \R^3$ such that
$\apr(B')=\tt AAN$. 
\end{ex}

The next example shows that if 
$B$ is an \textit{arbitrary} $(n-1) \times (n-1)$ symmetric matrix with 
$\apr(B)  = \tt SS \cdots S$ 
(i.e., the sequence each of whose terms is equal to $\tt S$),
then applying the aforementioned probabilistic techniques to 
$B$ need not yield an $n \times n$ (symmetric) matrix $B'$ with 
$\apr(B')=\tt SS \cdots SA$.

\begin{ex}\label{Bordering Example 2}\normalfont
Let
\[
B=
\mtx{
1 & 1 &  0 & 0\\
1 & 1 &  1 & 1\\
0 & 1 &  1 & 1\\
0 & 1 &  1 & 1} \in \F^{4\times4}
\mbox{\quad and \quad}
B'=
\mtx{
     B     & \vec{y} \\
\vec{y} &    t} \in \F^{5 \times 5},
\]
where $\F$ is an arbitrary field and $\vec{y}$ and $t$ are arbitrary.
Observe that $\apr(B)=\tt SSS$. 
We now show that there is no $\vec{y} \in \F^4$ such that
$\apr(B')=\tt SSSA$.
This is readily seen by noting that
$\det(B'[\{1,2,3,4\},\{2,3,4,5\}])=0$ for all $\vec{y} \in \F^{4}$
(two of the columns of $B'[\{1,2,3,4\},\{2,3,4,5\}]$ are the same). 
%Thus, there is no $\vec{y}$ in $\F^{4}$ such that
%$\apr(B')=\tt SSSA$. 
\end{ex}

%%%%%%%%%%%%%%%%%%%%%%%%%%%%%%%
%%%%%%%%%%%%%%%%%%%%%%%%%%%%%%%
%%%%%%%%%%%%%%%%%%%%%%%%%%%%%%%
%%%%%%%%%%%%%%%%%%%%%%%%%%%%%%%
%%%%%%%%%%%%%%%%%%%%%%%%%%%%%%%
%%%%%%%%%%%%%%%%%%%%%%%%%%%%%%%
%%%%%%%%%%%%%%%%%%%%%%%%%%%%%%%
%%%%%%%%%%%%%%%%%%%%%%%%%%%%%%%
\section{The ap-rank of a symmetric matrix}\label{s: ap-rank}
$\null$
\indent
This section is devoted to studying the ap-rank of a
symmetric matrix over an arbitrary field $\F$.
We start with basic observations.

\begin{obs}
Let $n \geq 2$ and $B \in \Fnn$ be symmetric.
Then $\aprank(B)$ is equal to
the index of the last $\tt A$ or $\tt S$ in $\apr(B)$.
\end{obs}

\begin{obs}
Let $B \in \Fnn$ be symmetric.
Then $\aprank(B) \leq \rank(B)$.
\end{obs}

\begin{obs}\label{aprank=0}
Let $B \in \Fnn$ be symmetric.
Then $\aprank(B) = 0$ if and only if $B$ is a diagonal matrix.
\end{obs}

Since the inverse of a
nonsingular non-diagonal matrix is non-diagonal,
the following fact is deduced easily
from the relationship between a matrix and its adjoint.

\begin{prop}\label{aprank nonsingular}
Let $B \in \Fnn$ be symmetric, non-diagonal and non-singular.
Then $\aprank(B) = n-1$.
\end{prop}

As we saw earlier (in Theorem \ref{thm: rank of a symm mtx}),
the rank of a symmetric matrix is equal to the order of a
largest nonsingular principal submatrix ---which led us to
call the rank of such a matrix ``principal.''
A natural question one should ask is whether an
analogous connection exists between
the rank of a symmetric matrix and
the order of a largest nonsingular almost-principal submatrix;
that is, is it the case that the
rank and ap-rank of a symmetric matrix is the same?
Obviously, the answer is negative,
since, for example, for a nonzero diagonal matrix $B$,
$\aprank(B) = 0$, while $\rank(B)>0$.
Moreover, since for an $n \times n$ matrix $B$
we must have $\aprank(B) \leq n-1$,
$\aprank(B) \neq \rank(B)$ if $B$ is nonsingular.
But what can we say if $B$ is non-diagonal and singular?
After establishing the following two lemmas, we show that
if $B$ is symmetric, non-diagonal and singular, and
does not contain a zero row, then $\aprank(B) = \rank(B)$.

\begin{lem}\label{zero row 1}
Let $B \in \Fnn$ be symmetric and singular.
Suppose that $\rank(B) = r$ and $B$ does not contain a zero row.
Let $B[\alpha]$ be an $r \times r$ nonsingular (principal) submatrix of $B$.
Then there exists $p \in \{1,2, \dots, n\} \setminus \alpha$
such that $B[\alpha \cup \{p\}]$ does not contain a zero row.
\end{lem}

\bpf
Without loss of generality, we may assume that
$\alpha = \{1,2, \dots, r\}$.
Suppose to the contrary that
the matrix $B[\alpha \cup \{p\}]$ contains a zero row
for all $p \in \{1,2, \dots, n\} \setminus \alpha$
(since $B$ is singular,
$\{1,2, \dots, n\} \setminus \alpha$ is nonempty).
It follows that $B=B[\alpha] \oplus C$,
where $C$ is an $(n-r) \times (n-r)$
matrix with zero diagonal.
Now, observe that
$\rank(B) = \rank(B[\alpha]) + \rank(C)$.
Then, as $\rank(B[\alpha]) = \rank(B)$, it follows that
$\rank(C) = 0$, implying that $C= O_{n-r}$,
which contradicts the fact that $B$ does not contain a zero row.
\epf

\begin{lem}\label{zero row 2}
Let $B \in \Fnn$ be symmetric.
Suppose that
$\apr(B) = a_1 a_2 \cdots a_{n-2}\tt N$ and
$\rank(B) = n-1$.
Then $B$ contains a zero row.
\end{lem}

\bpf
Suppose that $B=[b_{ij}]$.
Since $\apr(B)$ ends with $\tt N$,
the adjoint of $B$ is a diagonal matrix, and is of rank one, since $\rank(B)=n-1$.
Hence it follows that $Be_i=0$ for some standard basis vector $e_i$. 
Thus, the $i$th column (and hence row) 
of $B$ is zero. 
\epf

\begin{thm}\label{aprank singular}
Let $B \in \Fnn$ be a symmetric, non-diagonal, singular matrix
not containing a zero row.
Then $\aprank(B) = \rank(B)$.
\end{thm}

\bpf
Since $B$ is a non-diagonal matrix, $n \geq 2$.
Let $r=\rank(B)$.
Since $\aprank(B) \leq r$,
it suffices to show that $B$ contains a
nonsingular $r \times r$ almost-principal submatrix.
Since the rank of $B$ is principal,
$B$ contains a nonsingular $r \times r$ principal submatrix,
say, $B[\alpha]$.
By Lemma \ref{zero row 1},
there exists $p \in \{1,2, \dots, n\} \setminus \alpha$ such that
$B' := B[\alpha \cup \{p\}]$ does not contain a zero row.
Since $\rank(B) = r$,
and because $B'$ contains the
nonsingular $r \times r$ matrix $B[\alpha]$,
$\rank(B') = r$.
Let $\apr(B') = a'_1a'_2 \cdots a'_r$.
Since $B'$ is a singular $(r+1) \times (r+1)$ (symmetric) matrix
with $\rank(B') = r$, and because $B'$ does not contain a zero row,
Lemma \ref{zero row 2} implies that
$a'_r \neq \tt N$.
Hence, $B'$ contains a nonsingular, $r \times r$, almost-principal submatrix.
Then, as every almost-principal submatrix of $B'$ is also an
almost-principal submatrix of $B$, the desired conclusion follows.
\epf

Although the rank and ap-rank of a
symmetric matrix $B$ are not always the same,
the rank cannot exceed the ap-rank by
more than one if $B$ is a non-diagonal matrix:

\begin{thm}\label{ap-rank}
Let $B \in \Fnn$ be symmetric and non-diagonal. 
Define the parameter 
$t:=\max\{|\alpha|: B[\alpha] \mbox{ does not contain a zero row}\}$,
and
let $B'$ be the
$t\times t$ principal submatrix of $B$ not containing a zero row.
Then $\rank(B) - 1 \leq \aprank(B) \leq \rank(B)$.
Moreover,
$\aprank(B) = \rank(B)$
if and only if
$B'$ is singular.
Equivalently,
$\aprank(B) = \rank(B)-1$
if and only if
$B'$ is nonsingular.
\end{thm}

\bpf
Since $B$ is symmetric and non-diagonal,
it is immediate that $n \geq 2$ and $t \geq 2$.
Without loss of generality, we may assume that
$B' = B[1,2, \dots,t]$.
Thus, $B=B' \oplus O_{n-t}$.
Then, as $\rank(B) = \rank(B')$ and $\aprank(B) = \aprank(B')$,
it suffices to show that the desired conclusions hold for
the case with $B=B'$ (that is, the case with $t=n$);
thus, we assume that $B=B'$.
If $B$ is nonsingular, then, by Proposition \ref{aprank nonsingular},
$\aprank(B) = n-1 = \rank(B)-1$.
If $B$ is singular, then, by Theorem \ref{aprank singular},
$\aprank(B)=\rank(B)$.
It follows that
$\aprank(B)= \rank(B)-1$ or $\aprank(B)=\rank(B)$, implying that
$\rank(B) - 1 \leq \aprank(B) \leq \rank(B)$, as desired.

The remaining two statements, and their equivalency, is established easily using the above arguments in this proof and the fact that
$\rank(B) - 1 \leq \aprank(B) \leq \rank(B)$.
\epf

Although for a given symmetric matrix $B \in \Fnn$ we must have
$0\leq \rank(B) - \apr(B) \leq 1$ if $B$ is non-diagonal,
$\rank(B) - \apr(B)$ can attain any integer value on
the closed interval $[0,n]$ if $B$ is a diagonal matrix, since, for example,
$\rank(B) - \apr(B) = r$ if $B$ is a diagonal matrix with $\rank(B) = r$.

%%%%%%%%%%%%%%%%%%%%%%%%%%%%%%%
%%%%%%%%%%%%%%%%%%%%%%%%%%%%%%%
%%%%%%%%%%%%%%%%%%%%%%%%%%%%%%%
%%%%%%%%%%%%%%%%%%%%%%%%%%%%%%%
%%%%%%%%%%%%%%%%%%%%%%%%%%%%%%%
%%%%%%%%%%%%%%%%%%%%%%%%%%%%%%%
%%%%%%%%%%%%%%%%%%%%%%%%%%%%%%%
%%%%%%%%%%%%%%%%%%%%%%%%%%%%%%%
\section{Concluding remarks}\label{s: final}
$\null$
\indent
Given that the only difference between 
the epr- and apr-sequence is that 
the former depends on principal minors, 
while the latter depends on almost-principal minors,
it is worthwhile to compare the state of affairs for 
epr- and apr-sequences.
Although the apr-sequence was just introduced (in the present paper),
we already have a better understanding of 
this sequence than of the epr-sequence:
The epr-sequences of symmetric matrices over 
the field of order $2$ were 
completely characterized in \cite{XMR-Char 2};
however, for any other field, no such characterization exists. 
In Section \ref{s: No As}, 
the sequences not containing any $\tt A$s that 
are realized as the apr-sequence of a 
symmetric matrix over an arbitrary field $\F$ were 
completely characterized.
Moreover, in Section \ref{s: No As},
a necessary condition for a sequence to
be the apr-sequence of a symmetric matrix over a field $\F$ was presented.
It is clear, then, that our understanding of 
apr-sequences is already better than that of epr-sequences.

As stated in Section \ref{s: intro},
one of our motivations for introducing the
ap-rank and apr-sequence of a symmetric matrix was
answering Question \ref{qpr question},
which asks if we should attribute the fact that
neither $\tt NA$ nor $\tt NS$ can occur as a subsequence of
the qpr-sequence of a symmetric matrix $B \in \Fnn$ entirely to
the dependence of qpr-sequences on almost-principal minors.
The following remark answers Question \ref{qpr question},
under the assumption that $n \geq 3$
(the question is trivial when $n \leq 2$).

\begin{rem}
\normalfont
Let $n \geq 3$, $\F$ be a field and $B \in \Fnn$ be symmetric.
Then the fact that
neither $\tt NA$ nor $\tt NS$ can occur as a subsequence of 
$\qpr(B)$ is attributed entirely to the dependence of 
$\qpr(B)$ on almost-principal minors 
if and only if
$B$ is a non-diagonal  matrix for which there does not exist a generalized permutation matrix such that 
$P^TBP = T^p_q$ when $p \geq 2$, 
where
$T^p_q$ is the $n \times n$ matrix
\[
T^p_q:=\underbrace{A(K_2) \oplus A(K_2) \oplus \cdots \oplus A(K_2)}_{\mbox{$p$ times}} \oplus O_{q}
\in \Fnn.
\]
We now establish the previous statement. 
%%Let $B \in \Fnn$ be symmetric,
%%where $n \geq 3$ and $\F$ is an arbitrary field.
First, suppose that $B$ is a diagonal matrix with $\rank(B)=r$.
If $B$ is nonsingular, then
$\epr(B) = \tt AAA\overline{A}$.
If $B$ is singular, then
$\epr(B) = \tt \overline{S} \hspace{0.4mm}\overline{N}$,
with $\tt S$ occurring $r$ times and $\tt N$ occurring $n-r$ times.
Since $r$ is equal to the index of the last $\tt A$ or $\tt S$ in $\qpr(B)$
(see Observation \ref{qpr rank}),
neither $\tt NA$ nor $\tt NS$ is a subsequence of $\qpr(B)$,
regardless of what $\apr(B)$ is;
hence, the fact that neither $\tt NA$ nor $\tt NS$ is a subsequence of 
$\qpr(B)$ is not attributed entirely to the dependence of 
$\qpr(B)$ on almost-principal minors.

Now, suppose that $B$ is a non-diagonal matrix for which 
there exists a generalized permutation matrix $P$ such that 
$P^TBP=T^p_q$ for some $p \geq 2$.
Then either
$\epr(B)=\tt NS\overline{NS}NA$ (if $q=0$) or
$\epr(B)=\tt NS\overline{NS}\hspace{0.4mm}\overline{N}$ (if $q \geq 1$),
with $\tt \overline{N}$ containing $q$ copies of $\tt N$.
Moreover,
$\apr(B) = \tt SNS\overline{NS} \hspace{0.4mm} \overline{N}$,
with $\tt \overline{N}$ containing $q$ copies of $\tt N$
(see Proposition \ref{SN}).
Then, as $\rank(B)=n-q$,
and because $\rank(B)$ is equal to
the index of the last $\tt A$ or $\tt S$ in $\qpr(B)$,
$\qpr(B) = \tt SSS\overline{S} \hspace{0.4mm} \overline{N}$.
It is easy to see that the fact that neither $\tt NA$ nor $\tt NS$ is a
subsequence of $\qpr(B)$ is attributed to {\em both}
the principal and the almost-principal minors of $B$.

Finally, 
suppose that $B$ is a non-diagonal matrix for which 
there does not exist a generalized permutation matrix such that 
$P^TBP = T^p_q$ when $p \geq 2$.
Let $\qpr(B)=q_1q_2 \cdots q_n$.
It suffices to present an argument based solely on
almost-principal minors for the fact that
if $q_k= \tt N$ for some $k$,
then $q_j = \tt N$ for all $j \geq k$;
we present one based on the ap-rank and apr-sequence of $B$:
Suppose that $q_k= \tt N$ for some $k$.
Let $\apr(B)=a_1a_2 \cdots a_{n-1}$.
If $k =n$, then there is nothing to prove;
thus, assume that $k \leq n-1$.
Obviously, $a_k = \tt N$.
We now show that
$\apr(B)$ does not contain $\tt NA$ nor $\tt NS$ as a subsequence.
If it was the case that $\apr(B)$ contained
$\tt NA$ or $\tt NS$ as a subsequence, then
Theorem \ref{necessary condition} would imply that
$\apr(B) = \tt SNS\overline{NS} \hspace{0.4mm} \overline{N}$,
and, then, Proposition \ref{SN} would imply that
there exists a generalized permutation matrix such that
$P^TBP=T^p_q$ for some $p \geq 2$, leading to a contradiction.
Hence, neither $\tt NA$ nor $\tt NS$ is a subsequence of $\apr(B)$.
It follows that $a_j = \tt N$ for all $j \geq k$, and, therefore, that
$\aprank(B) \leq k-1$.
Then, as $B$ is non-diagonal,
Theorem \ref{ap-rank} implies that $\rank(B) \leq k$.
Since $\rank(B)$ is equal to the index of
the last $\tt A$ or $\tt S$ in $\qpr(B)$ (see Observation \ref{qpr rank}),
$q_j = \tt N$ for all $j \geq k+1$.
Then, as $q_k = \tt N$, the desired conclusion follows.
\end{rem}

%%%%%%%%%%%%%%%%%%%%%%%%%%%%%%%
%%%%%%%%%%%%%%%%%%%%%%%%%%%%%%%
%%%%%%%%%%%%%%%%%%%%%%%%%%%%%%%
%%%%%%%%%%%%%%%%%%%%%%%%%%%%%%%
%%%%%%%%%%%%%%%%%%%%%%%%%%%%%%%
%%%%%%%%%%%%%%%%%%%%%%%%%%%%%%%
%%%%%%%%%%%%%%%%%%%%%%%%%%%%%%%
%%%%%%%%%%%%%%%%%%%%%%%%%%%%%%%
\subsection*{Acknowledgments}
$\null$
\indent
The research of the first author was supported in part by an
NSERC Discovery Research Grant RGPIN-2014-06036.
The authors express their gratitude to 
a referee for their careful review, and for 
many helpful and constructive comments on a 
previous version which greatly improved the presentation of the paper.
Moreover, they thank a referee for bringing 
Example \ref{Bordering Example 1} to their attention.
%%%%%%%%%%%%%%%%%%%%%%%%%%%%%%%
%%%%%%%%%%%%%%%%%%%%%%%%%%%%%%%
%%%%%%%%%%%%%%%%%%%%%%%%%%%%%%%
%%%%%%%%%%%%%%%%%%%%%%%%%%%%%%%
%%%%%%%%%%%%%%%%%%%%%%%%%%%%%%%
%%%%%%%%%%%%%%%%%%%%%%%%%%%%%%%
%%%%%%%%%%%%%%%%%%%%%%%%%%%%%%%
%%%%%%%%%%%%%%%%%%%%%%%%%%%%%%%

\end{document}